\documentclass[11pt]{amsart}

\usepackage{amssymb,latexsym,amsmath,amscd,epsfig,amsthm, amsfonts}

\theoremstyle{theorem}
\newtheorem{lemma}{Lemma}
\newtheorem{theorem}{Theorem}

\newtheorem{proposition}{Proposition}

\numberwithin{equation}{section}
\theoremstyle{remark}

\newcommand{\T}{{\mathbf T}}
\newcommand{\D}{{\mathbf D}}

\newcommand{\R}{{\mathbf R}}
\newcommand{\Harm}{{\rm Harm}}
\newcommand{\const}{{\rm const}}
\newcommand{\conv}{{\rm conv}}

\newcommand{\HH}{{\mathcal K}}

\newcommand{\A}{{\mathcal A}}
\newcommand{\I}{{\mathcal I}}
\newcommand{\M}{{\mathcal M}}
\newcommand{\RR}{{\mathcal R}}
\newcommand{\beq}{\begin{equation}}
\newcommand{\eeq}{\end{equation}}
\newcommand{\beqs}{\begin{equation*}}
\newcommand{\eeqs}{\end{equation*}}
\newcommand{\beg}{\begin{gather}}
\newcommand{\eeg}{\end{gather}}

\begin{document}

\dedicatory{Dedicated to Victor Petrovich Havin on the occasion of his 75th birthday}

\author {A. Borichev}
\address{Universit\'e Aix-Marseille, 39, rue Joliot Curie, 13453, Marseille Cedex
13, France}
 \email{borichev@cmi.univ-mrs.fr}
 
 \author{Yu. Lyubarskii}
\address{Department of Mathematical Sciences, Norwegian University of Science and Technology, NO-7491, Trondheim, Norway}
\email{yura@math.ntnu.no} 
\author{E. Malinnikova}
\address{Department of Mathematical Sciences, Norwegian University of Science and Technology, NO-7491, Trondheim, Norway}
\email{eugnia@math.ntnu.no} 
\author{P. Thomas}
 \address{Universit\'e Paul Sabatier, 31062 Touluose Cedex 9, France}
\email{pthomas@math.univ-toulouse.fr}

\title
{Radial growth of functions from the Korenblum space} 


\thanks{A.B. was partially supported by the ANR project DYNOP,
Yu.L. and E.M.  were partly supported by the Research
Council of Norway, grants 160192/V30 and  177355/V30.}   

\begin{abstract}
 We study radial behavior  of analytic  and harmonic  functions,  
 which admit a certain majorant in the unit disk. We prove that extremal growth or decay may occur only along small sets of radii and give precise estimates of these exceptional sets.  
 \end{abstract}
\maketitle

\section{Introduction}

We study radial behavior of analytic and harmonic functions in the unit disc. 
In order to describe the problem let us start with the classical results of Lusin and Privalov, see e.g. \cite{Privalovbook}  Ch. IV. 

\medskip

\noindent {\bf Theorem A.} (Lusin, Privalov)  {\em Let $f(z)$ be a function 
analytic in the unit disc $\D$ and $E$ be a subset of the unit circle $\T$ of positive linear measure. If $f$ tends to zero non-tangentially at each
point of $E$, then $f=0$.}

\medskip

The situation changes if one considers radial limits.  

\medskip 


\noindent{\bf Theorem B.} (Lusin, Privalov)    {\em There exists an analytic function $f$ in $\D$ such that $\lim_{r\rightarrow 1}f(re^{i\phi})=0$ for almost
every $\phi\in[0,2\pi)$.}

\medskip

These results can be reformulated for harmonic functions. 
The first theorem says that there are no $u\in \Harm(\D)$ that tends to $+\infty$ 
non-tangentially on a set of positive measure in $\T$, while the second gives 
a function in $\Harm(\D)$ that tends radially to $+\infty$ almost everywhere on $\T$ (we remark that the function $f$ in Theorem B  can be chosen zero-free). We refer also to \cite{BS,KK} for other relevant examples.  
The growth of harmonic functions tending radially to $+\infty$ almost everywhere
can be arbitrarily slow: the statement below is a special case of a theorem in \cite{KK}.  


\medskip

\noindent{\bf Theorem C.}(Kahane, Katznelson) { \em Let $v(r)$ be a positive increasing function on $[0,1)$ and $\lim_{r\rightarrow 1-}v(r)=\infty$. Then there exists  $u\in  \Harm(\D)$
such that  
\beq
\label{eq:30}
u(z)\le v(|z|)
\eeq 
and $\lim_{r\rightarrow 1-} u(re^{i\phi})=\infty$ for a.e. $\phi\in[0,2\pi)$. }

\medskip 

In this article we address the following questions. 

{\em   Let the function $v$ be as above and $u\in \Harm(\D)$ satisfy
\eqref{eq:30}.}

{\em $\bullet$ How fast (with respect to $v$) can $u$ grow to $+\infty$  
along massive sets of radii?} 

{\em $\bullet$ How fast (with respect to $v$) can $u$ decay to $-\infty$ along massive sets of radii?}

\medskip

We restrict ourselves to a particular majorant function
\[
v(r)=\log \frac 1 {1-r}.
\]
This choice is motivated by its relevance to the classical Korenblum space
$A^{-\infty}$ (see \cite{K}). It  also serves  as a model case for more general majorants.

The typical answer to the first question is that at almost all radii the function $u$   grows (if it grows at all) slower that $v$, the exceptional set has zero Lebesgue measure.
 We give precise estimates on the size of exceptional sets in terms of the Hausdorff measures with respect to the scale of functions 
$h_\alpha(t)=t  |\log t|^\alpha$, $\alpha > 0$.

Regarding the second question we first remark that  the function $u$ that satisfies \eqref{eq:30}  may decay to $-\infty$ along  radii much faster than 
$-v(r)$, so the harmonic function $-u$ may fail to satisfy  \eqref{eq:30}. 
However, given $M(s)$, $s>0$ such that $M(s)/s \to +\infty$, as $s\to +\infty$, the 
set $\{z\in \D: u(z)< -M(v(|z|)\}$  is small  (sharp estimates for typical $M$ are 
obtained in \cite{BL}). We show that along most radii $-u$ grows slower than $v$, the estimates of the exceptional set being the same as in the answer to the first question. For the maximal possible decay of harmonic functions throughout the whole disc see \cite{N}, \cite{B} and references therein.
  

Our statements can be reformulated for zero-free functions from 
the Korenblum class. Now the second statement describes how fast an analytic function can  approach zero along some radii. Actually this statement holds true for 
 any (non necessarily  zero-free)  function from $A^{-\infty}$. 
 At the same time adding zeros may result in extremal radial growth 
along almost all radii.

 
The paper is organized as follows. The next section includes  
definitions and formulations of the main results in terms of analytic functions. 
In Section 3 we deal with harmonic (subharmonic)  functions. Using standard 
estimates of the Poisson integral we show that fast radial growth (decay) implies 
non-tangential growth (decay) and thus may occur only on a set of zero 
measure. The main results are proved in Section 4: departing from non-tangential growth (decay) and   harmonic measure inequalities we obtain more precise estimates of the size of the exceptional sets. These estimates are sharp, as shown by  examples collected in Section 5. 
We also give an example which shows that the situation becomes very different if one considers growth just along  sequences of points: there exists a function such that
{\em every}  radius  
contains a sequence of points of extremal growth.
In Section 6 we consider {\em positive} harmonic functions satisfying \eqref{eq:30} and show that for such functions the answer to the first question is different.  
Finally, Section 7 contains a theorem about Hausdorff measure of Cantor-type sets.


\section{Formulation of the main results}

An analytic function $f$ in $\D$ is said to be of class
$A^{-\infty}$ if there exist constants $C$ and $k$ such that
\[
|f(z)|\le\frac{C}{(1-|z|)^k}.
\]

For a function $f\in A^{-\infty}$ we define
\beq
\label{eq01a}
D_+(f)=\left\{\theta\in[0,2\pi): \liminf_{r\rightarrow
1}\frac{\log |f(re^{i\theta})|}{|\log(1-r)|}>0\right\},
\eeq
\beq  
\label{eq02a}
D_-(f)=\left\{\theta\in[0,2\pi): \limsup_{r\rightarrow
1}\frac{\log|f(re^{i\theta})|}{|\log(1-r)|}<0\right\}.
\eeq

We recall the definition of the Hausdorff measure.
Given an increasing  function 
$\lambda$, $\lambda:[0,1)\rightarrow [0,+\infty)$, $\lambda(0)=0$,  
we denote by $H_\lambda(C)$ the corresponding Hausdorff measure of a set $C\subset \T$ (or $C\subset \R$), which is defined as 
\beqs
\label{eq05}
H_\lambda(C)=\lim_{\epsilon\rightarrow 0}\inf\left\{\sum_s \lambda(|J_s|) : C\subset\cup_s J_s,\ |J_s|<\epsilon\right\},
\eeqs
here $J_s$ are arcs of $\T$  (respectively intervals of $\R$).

The main results of the paper give estimates on the size of the sets $D_\pm(f)$.

\begin{theorem}
\label{thm:precise}
Let $\lambda(t)=o(t|\log t|^\omega)$, $t\to 0$, for every $\omega >0$. Then\\
{\it{(i)}}  ${H}_\lambda(D_+(f))=0$ for each zero-free $f\in A^{-\infty}$;\\
{\it{(ii)}} ${H}_\lambda(D_-(f))=0$ for each $f\in A^{-\infty}$.
 \end{theorem}
 
These results are sharp  as follows from the next statement. 

\begin{proposition}
\label{pr:example}
For any $\alpha>0$ there exists a zero-free function $f\in A^{-\infty}$ such that $f^{-1}\in A^{-\infty}$ and ${H}_\lambda(D_+(f))=\infty$  for $\lambda(t)=t|\log t|^\alpha.$
\end{proposition}
Note that for zero-free functions we have $D_+(f)=D_-(f^{-1})$; thus, 
Proposition~\ref{pr:example}
shows that our condition on $\lambda$ is precise in both assertions of 
Theorem~\ref{thm:precise}.

\smallskip

There are no analogues of the first estimate in Theorem~\ref{thm:precise} for general functions from $A^{-\infty}$.  This can be seen by analyzing functions having "regular" growth in $\D$ like those given in Theorem 2 in
\cite{Seip}. The argument in \cite{Seip} relies on the atomization techniques, 
and in this paper we use a simple explicit function constructed by Horowitz.

Given a number $\mu >1$ and an integer $\beta >1$, consider the function
\beq
\label{eq:07a}
f_{\mu,\beta}(z)= \prod_{k=1}^\infty \left (  1+\mu z^{\beta^k}
                                                      \right ).
\eeq                                                       
It follows from \cite{Horowitz} (see also \cite{K}) that $f_{\mu,\beta}\in A^{-\infty}$. 
 
\begin{proposition}
\label{prop:horowitz} Let the numbers $\mu$ and $\beta$ 
satisfy the conditions 
\beq
\label{eq:07b}
1-\frac1{\mu^{1-1/\sqrt{\beta}}}  \geq  \frac 1 e  , \quad \mu > e.
\eeq
Then
$D_+(f_{\mu,\beta})$  has full measure in $\T$.
\end{proposition} 

\smallskip

We also consider  the extremal growth on subsets of radii.
Given a function $f\in A^{-\infty}$, denote
\[
G_+(f)=\left \{\theta\in [0,2\pi): \limsup_{r\to 1}\frac{\log |f(re^{i\theta})|}
{|\log(1-r)|}>0 \right \}
\]
Clearly $G_+(f)\supset D_+(f)$.

\begin{proposition}
\label{prop:3}There exists a zero-free $f\in A^{-\infty}$, such that 
$G_+(f)=\T$.
\end{proposition} 

\smallskip

The estimate in the first statement in Theorem~\ref{thm:precise}   can be improved if we assume that $|f|$ is bounded from below by a positive constant. 
This improvement corresponds to the difference between estimates of the Poisson integral with respect to a premeasure (as in Theorem~\ref{thm:precise})
and a Borel measure as in Theorem~\ref{thm:positive}  below; we refer the reader to \cite{K} for the definition and basic properties of premeasures.

\begin{theorem}
\label{thm:positive}
Let $\lambda(t)=t|\log t|$. Suppose that $f\in A^{-\infty}$ and $|f|$ is bounded from below by a positive constant. Then the set  $G_+(f)$ 
is a countable union of sets with finite $ {H}_\lambda$ measure.
 
There exists $f\in A^{-\infty}$, such that $|f|$ is bounded form below by a positive constant and ${H}_\lambda(D_+(f))=\infty$.
\end{theorem}

\smallskip

In order to construct examples in Proposition~\ref{pr:example} 
and Theorem~\ref{thm:positive} we 
use  Cantor-type  sets  having the following structure:
\[
C=\cap_s C_s,  \ C_s \supset C_{s+1}, \ C_0=[0,1], 
\]
each set $C_s$ is a union of $N_s$  segments $\{I^{(s)}_j\}_j$ of the same length $l_s$. For each such segment the intersection $C_{s+1}\cap I^{(s)}_j$
is a union of $k_s$  disjoint segments of length $l_{s+1}$. We assume, of course, that
\beqs
\label{eq:34}
l_s \searrow 0, \ s \to \infty; \ k_sl_{s+1} < l_s, \ \text{and}  \  N_s=k_0k_1\ldots k_{s-1}.
\eeqs

\begin{theorem}
\label{th:cantor}
Let   $\lambda:[0,1)\rightarrow[0,+\infty)$ be a continuous increasing function with $\lambda(0)=0$, such that  for some $a>0$ and $s>s_0$
\begin{equation}
\label{eq:cantorh} 
\frac{\lambda(l)}{l}\ge a\frac{\lambda(l_{s+1})}{l_{s+1}}\quad{\rm{
for\ any}}\ l\in[l_{s+1},l_s).
\end{equation}
Then
\beq
\label{eq07}
\liminf_{s\rightarrow\infty} N_s\lambda(l_s) 
                   \ge H_\lambda(C)
                           \ge \frac{a}{2}\liminf_{s\rightarrow\infty} N_s\lambda(l_s).
\eeq
 \end{theorem}

Other results of such type are given in \cite{E, Mo}; unfortunately, 
we are not able to apply those results in our situation. We prove 
Theorem~\ref{th:cantor}  in the last section and believe that it may be of its own interest. 

\smallskip


\section{From radial to non-tangential growth}

To deal with zero-free functions  from $A^{-\infty}$ we consider the corresponding class of harmonic functions. 
A   function $u \in \mbox{Harm}(\D)$ is said to be  of class $\HH$ if
there exists a constant $C$ such that
\[
u(z)\le C\log\frac{e}{1-|z|},\qquad z\in \D.
\]
If $f\in A^{-\infty}$ is a zero-free function,  then
$u(z)=\log|f(z)|$ belongs to  $\HH$.

Given a function $u\in \HH$ we denote 
\beqs
\label{eq01}
E_+(u)=\left\{\theta\in[0,2\pi): \liminf_{r\rightarrow
1}\frac{u(re^{i\theta})}{|\log(1-r)|}>0\right\},
\eeqs
\beqs
\label{eq02}
E_-(u)=\left\{\theta\in[0,2\pi): \limsup_{r\rightarrow
1}\frac{u(re^{i\theta})}{|\log(1-r)|}<0\right\}.
\eeqs

In this section we first show that fast radial
growth along radii implies fast non-tangential growth.
 We use the standard notation
\[
P(re^{i\theta})=\frac1{2\pi}\frac{1-r^2}{1-2r\cos\theta+r^2}.
\] 
for the Poisson kernel.  

Let $r\in (0,1)$, $\tau \in (0,1)$, and   $0<\delta<\tau(1-r)$.
Then
\beq
\label{eq:31}
P(re^{i\theta})>(1-\tau)P(re^{i(\theta+\delta)})
 \eeq 
This inequality can be proved by  elementary calculations, it can be also 
viewed as a special case  of the Harnack inequality.
 
\begin{lemma}
\label{l:cone}
Let $u\in \HH$, $C=\sup \{ u(z) \left (  \log  e/ (1-|z|)  \right ) ^{-1};z\in\D\}$. Suppose that for some $\sigma >0$,
$\theta\in [0,2\pi]$, and $r\in(0,1)$ 
\beqs
\label{eq:08}
u(re^{i\theta}) > \sigma \log \frac e {1-r} .
\eeqs
Then  
\beq
\label{eq:09}                                                        
u(re^{i(\theta+\delta)}) >\frac  \sigma  2 \log \frac e {1-r}                                              
\eeq  
for $|\delta|< \tau_1 (1-r)$, where $\tau_1=\tau_1(C,\sigma)>0$.
   \end{lemma}
   
\begin{proof} Let $R= (1+r)/2$.
We apply \eqref{eq:31}, replacing $P(r, \cdot)$ by $P(r/R, \cdot)$ and assuming that
$|\delta| < \tau 
\left ( 1- \frac r R \right ) $ with $\tau<1$. We obtain
\begin{gather*}
u(re^{i\theta})=\int_0^{2\pi}u(Re^{i\phi})P\left(\frac{r}{R}e^{i(\theta-\phi)}\right)d\phi=
(1-\tau)u(re^{i(\theta+\delta)})+\\
\int_0^{2\pi } u(Re^{i\phi})\left(P\left(\frac{r}{R}e^{i(\theta-\phi)}\right)-
(1-\tau) P\left(\frac{r}{R}e^{i(\theta+\delta-\phi)}\right)\right)d\phi<\\
\qquad (1-\tau)u(re^{i(\theta+\delta)})+C\tau\log\frac{e}{1-R}.
\end{gather*}
Hence 
\[
\sigma \log \frac e {1-r} < (1-\tau)u(re^{i(\theta+\delta)}) + C \tau \log \frac  e{1-r} 
    + C \tau \log 2.
\]
Taking $\tau$ small enough we now obtain relation  \eqref{eq:09}  with $\tau_1=\tau/2$. 
 \end{proof}

The proof of Lemma~\ref{l:cone} works also for the radial decay; however, 
for  this case we need  a more general setting involving subharmonic functions.

\begin{lemma}
\label{l:angle}
Let $v$ be a subharmonic function on $\D$ and
\[
v(z)\leq C \log \frac e {1-|z|}, \ z\in \D.
\]
 Suppose that for some $\sigma >0$,
$\theta\in [0,2\pi]$, and $r_0\in(0,1)$ 
\beqs
\label{eq:10}
v(re^{i\theta}) < - \sigma \log \frac e {1-r}, \ r>  r_0  
\eeqs                                                      
Then
\beqs
\label{eq:11}                                                        
v(re^{i(\theta+\delta)}) < -\frac  \sigma  2 \log \frac e {1-r}, \ r> r_1                                                        
\eeqs  
for $|\delta|< \tau_2 (1-r)$, where $\tau_2=\tau_2(C,\sigma)>0$,
$r_1=r_1(r_0)<1$.
\end{lemma}

\begin{proof}

Without loss of generality assume that $\theta=0$. Consider the function 
\[
w(\rho e^{i\varphi})= v(1-\rho e^{i \varphi})
\]
which is subharmonic in the  domain 
\[
G= \{\zeta=\rho e^{i\varphi}: 0< \rho < 1, 0<\varphi < \pi/4\}.
\]
We have
\beq
\label{eq:11a}
w(s) < - \sigma \log \frac e s, \ \ s< \rho_0, \ \text{and} \
   w(s e^{i\varphi})< C \log \frac e s, \ s e^{i\varphi}\in G,
 \eeq   
and we need to prove that 
\beq
\label{eq:11b}
w(\rho e^{i\delta}) < - \frac \sigma 2 \log \frac e \rho, \ \rho < \rho_1 \ \text{for all}  \
\delta \in (0, \delta_1), \ \delta_1=\delta_1(\sigma, C).
\eeq

Consider an auxiliary function $u(\zeta)$  which is harmonic in the domain
$R=\{\zeta=\rho e^{i\varphi}: 1/4 < |\rho| < 1, \  0< \varphi < \pi/4\}$ and has the boundary  
values
\[
u(\rho)=0, \ \rho\in (1/4, 1), \ u(\zeta)|_{\partial R \setminus (1/4, 1)} = 1.
\]
Fix now $\rho < \rho_0/2$. It follows from \eqref{eq:11a}  that
\[
w(\zeta)+ \sigma \log \frac e {\rho}  \leq 
         (C+\sigma )  u\left (\frac \zeta {2\rho}\right )  \log \frac e \rho +\sigma\log 2 
                       ,\quad \zeta \in \partial (2\rho R).
\]
   Therefore
\[
w(\zeta) < -    \sigma \log \frac e {\rho}  +   
    (C+\sigma ) u\left (\frac \zeta {2\rho}\right )     \log \frac e \rho+ \sigma\log 2 
           ,\quad   \zeta \in  2\rho R.
       \]
In order to obtain \eqref{eq:11b}  it remains to take 
$ \zeta = \rho e^{i\varphi}$  and  note that $u(e^{i\varphi}/2)\to 0 $  as $\varphi  \to 0$.
\end{proof}

It follows from Lemmas~\ref{l:cone}, \ref{l:angle}, and the Lusin-Privalov theorem, that 
\newline\noindent
$|E_+(u)|=0$  for any $u\in \HH$  and also $|D_-(f)|=0$  for any $f\in A^{-\infty}$.


\section{Proof of Theorem \ref{thm:precise}}
 
In this section we prove  Theorem~\ref{thm:precise}. The first statement of the theorem is equivalent to the following

\medskip 

\label{thm:precise_harm}

\noindent {\bf {Theorem 1 {\em \!\!i}}}. {\em 
Let $\lambda(t)=o(t|\log t|^\omega)$, $t\to 0$, for every $\omega >0$.
Then ${H}_\lambda(E^+(u))=0$ for every $u\in\HH$. }

\begin{proof}
Fix $u\in \HH$.   Let
\[
E_{n}=\left \{
                    e^{i\theta}: u(re^{i\theta})\ge
    \frac1n \log \frac e {1-r},\ r\ge 1-\frac1n
               \right\}.
\]
Since $E_+(u)=\cup_nE_n$,
it suffices to prove that $H_\lambda(E_{n})=0$  for each $n$.
We use the standard cone construction. 
For $e^{i\theta}\in\T$ and $a<1$ consider the Stolz angle
$\Gamma^a_\theta=\conv(e^{i\theta}, a\D)$, i.e., the convex hull 
of $e^{i\theta}$ and the disc of radius $a$.
By Lemma~\ref{l:cone}, one can find $a>0$ such that 
$u(z)\ge\frac1{2n} \log \frac e {1-|z|}$ for all $e^{i\theta}\in E_{n}$ and  $z\in\Gamma_a^\theta,  \ |z|>1-\frac1n$. 

Let 
\[
\Omega=\cup_{\theta\in E_{n}}\Gamma_a^\theta.
\] 
The function  $u$ is bounded from below on $\Omega$; let, say,  
$u\ge c_0$. 
Let $t<1$ be sufficiently close to 1 and
\[
 \Omega_t=\Omega\cap t\D.
 \]
For an appropriate $b=b(a)$ we have
\[
\partial\Omega_t\cap t\T=tE_n^{b(1-t)}=\{te^{i\theta}: |\theta-\theta_0|<b(1-t) \  \text{for some}\ e^{i\theta_0}\in E_n\}.
\]

Estimating the subharmonic function $u$ in the domain $\Omega_t$, 
$t>1-\frac1n$, in terms of harmonic measure, we obtain
\begin{equation}
\label{eq:harmonic_measure}
u(0)\ge c_1
+\omega(0, tE_{n}^{b(1-t)}, \Omega_t) 
         \frac 1{2n} \log \frac e {1-t}.
\end{equation}
 Domains $\Omega_t$ have Lipschitz boundaries with Lipschitz constants 
bo\-un\-ded uniformly in $t$.  By the Lavrentiev theorem, (see e.g. \cite{GM}, Chapter VII, Theorem 4.3), 
there exist $c$ and 
$\gamma$  
such that,  for each subarc 
$I\subset \partial\Omega_t$  and  $A\subset I$,  we have 
\[
\frac{\omega(0, A, \Omega_t)}{\omega(0,I,\Omega_t)}\ge c\left(\frac{l(A)}{l(I)}\right)^\gamma,
\] 
here $l$ is the arc-length on $\partial\Omega_t$. In particular, by \eqref{eq:harmonic_measure},
\[
l(tE_n^{b(1-t)})^\gamma\le c^{-1}l(\partial \Omega_t)^\gamma\omega(0,tE_{n}^{b(1-t)},\Omega_t)\le C \left ( \log \frac e {1-t}  \right )^{-1},
\]
where $C=C(n)$ does not depend on $t$.
Hence for all $\epsilon>0$ small enough we have 
 \[
 l(E_n^\epsilon)\le 
            C\left(\log\frac{be}{\epsilon}\right)^{-1/\gamma}.
 \]
Therefore one  can cover $E_n$ by $N_\epsilon$ intervals of length 
$\epsilon$, with 
\[
N_\epsilon\le 2\epsilon^{-1}C \left(\log\frac{be}{\epsilon}\right)^{-1/\gamma}.
\] 
Then 
\[
H_\lambda(E_{n})\le\liminf_{\epsilon\rightarrow 0} N_\epsilon\lambda(\epsilon)\le
\liminf_{\epsilon\rightarrow 0}  2\epsilon^{-1}C
        \left(\log\frac{be}{\epsilon}\right)^{-1/\gamma}\lambda(\epsilon).
\]
The condition on $\lambda$ implies
${H}_\lambda(E_n)=0 $  and we are done. 

\end{proof}

To prove the second part of Theorem~\ref{thm:precise} we repeat the argument for the subharmonic function $v(z)=\log|f(z)|$, using Lemma~\ref{l:angle} instead of Lemma~\ref{l:cone}, and replace the inequality \eqref{eq:harmonic_measure} by the following estimate, valid for subharmonic functions:
\[
 v(0)\le \int_{\partial\Omega_t}v(z)d\omega(0,z,\Omega_t)\le c_1-\omega(0, tE_{n}^{b(1-t)}, \Omega_t) 
         \frac 1{2n} \log \frac e {1-t}.
         \]


\section {Sharpness of results}

First we construct  functions from $\HH$ with "large" sets of growth.

\begin{lemma}
\label{l:ex2}
For each integer $A\geq 2$ the series 
\[
u(z)=\Re \sum_{k=0}^\infty A^kz^{2^{A^k}}
\]
converges in $\D$ and $|u(z)|\le C\log\frac{e}{1-|z|}$.
\end{lemma}

\begin{proof}
Fix $z\in\D$ sufficiently close to the boundary,  
and choose $N$  such that
\beq
\label{eq:23}
	2^{-A^{N}}\ge 1-|z| > 2^{-A^{N+1}}.
\eeq
Then 
\[
-\log |z| = -\log(1-(1-|z|))\geq 1-|z| > 
 2^{-A^{N+1}},
\]
 and for $n\ge N+1$ we have

$$
 \frac {A^{n+1} |z|^{2^{A^{n+1}}} } 
  {A^n |z|^{2^{A^n}}} =
                 A|z|^{2^{A^{n+1}}-2^{A^n}} 
                 \le
       Ae^{  - 2^{-A^{N+1}}\left ( 
           2^{A^{n+1}} - 2^{A^n}
       \right  )    }  <\delta(A)<1,  
$$
with
\begin{equation}       
\lim_{A\to\infty}\delta(A)=0.    \label{eq:qq}     
\end{equation}

 Therefore
\begin{multline*}
 |u(z)|\le \sum_{n=0}^N A^n |z|^{2^{A^n}}  +
       \sum_{n=N+1}^\infty A^n |z|^{2^{A^n}} \leq  \\
     A^{N+1}+A^{N+1}\frac 1 {1-\delta(A)}  \leq
         \frac 2 {1-\delta(A)} \frac A {\log 2} \log \frac 1 {1-|z|}. 
 \end{multline*}
 \end{proof}
 
  \begin{proof}[Proof of Proposition \ref{pr:example}] 
 Let  $A$ be large enough and
 \[
 f(z)=\exp\left(\sum_{k=1}^\infty A^kz^{2^{A^k}}\right).
 \]
 Then
 $u=\log|f|$ is the function from the previous Lemma, hence  both $f$ and $f^{-1}$ are from $A^{-\infty}$.  
   If, for some $\phi$, we have $\cos(2^{A^k}\phi)\ge1/\sqrt {2}$ for each $k$, then
   $\Re \Bigl(\bigl(re^{i\phi}\bigr)^{2^{A^k}}\Bigr) \ge r^{2^{A^k}}/\sqrt{2}>0 $. Taking $N$ as in \eqref{eq:23}, with $z=re^{i\phi}$, we get
     \[
   u(re^{i\phi})\ge \frac12 A^N r^{2^{A^N}}\ge\frac18 A^N\ge \frac{1}{8A\log 2}\log\frac{1}{1-r},\]
   thus $\phi\in E_+(u)$.
Denote
\[
C_j=\cap_{k=0}^j\{\phi: \cos(2^{A^k}\phi)\ge1/\sqrt{2}\}, \ C= \cap C_j.
\]
Then $C_j$ is the union of $N_j$ intervals of length $l_j=c2^{-A^j}$, 
where $c$ is an absolute constant. We call them intervals from $j$-th generation. Each of them contains $k_{j+1}=c2^{A^{j+1}-A^j}$ intervals from the next generation.  So it is easy to see that 
$N_j=c^j2^{A^{j}}$. 
Theorem~\ref{th:cantor} with $\lambda(t)=t|\log t|^\alpha$ now yields
\[
H_\lambda(E_+(u))\ge H_\lambda(C)\ge \frac{c_1}{2A^\alpha}\liminf_j c^{j}A^{\alpha j}.
\]
We chose $A$ such that $A^\alpha>c^{-1}$ and obtain $H_\lambda(E_+(u))=+\infty$.
\end{proof}

Next we construct an auxiliary harmonic function.

\begin{lemma} There exist a function $h\in  {\rm Harm}(\D)$
and a positive $B$ such that $|h(z)|\le B|z|$, $z\in\D$, and
$$
\max_{1/6<r<1/3}h(re^{i\theta})\ge 1,\qquad \theta\in[0,2\pi).
$$
\end{lemma}

\begin{proof} Set $K=\{\frac{t+1}6 e^{3\pi ti}, 0\le t\le 1\}$. Let $f$ be a function
equal to $0$ in a small neighborhood of $0$ and to $2$ in a small neighborhood
of $K$. By the Runge theorem, we can approximate $f$ by a polynomial
$g$ in such a way that $|g-f|<\frac 13$ on $K\cup\{0\}$.
Then we can just set $h=\Re(g-g(0))$.
\end{proof}

Proposition \ref{prop:3} follows immediately from the following

\begin{lemma} If an integer A is large enough, then
the series 
$$
u(z)=\sum_{k=0}^\infty A^kh\bigl(z^{2^{A^k}}\bigr)
$$
converges in $\D$ to a function from $\HH$, and for some $d>0$,
$$
\limsup_{r\to 1}\frac{u(re^{i\theta})}{|\log(1-r)|}\ge d,\qquad \theta\in[0,2\pi).
$$
\end{lemma}

\begin{proof}
Fix $z\in\D$ sufficiently close to the boundary,  
and choose $N$  such that
$$
2^{-A^{N}}\ge 1-|z| > 2^{-A^{N+1}}.
$$
By \eqref{eq:qq},
\begin{multline*}
 |u(z)|\le B\sum_{n=0}^N A^n |z|^{2^{A^n}}  +
       B\sum_{n=N+1}^\infty A^n |z|^{2^{A^n}} \le  \\
     A^{N+1}B+A^{N+1}B\frac 1 {1-\delta(A)}  \le
         \frac 2 {1-\delta(A)} \frac {AB} {\log 2} \log \frac 1 {1-|z|}. 
\end{multline*}
The same estimate gives
$$
|u(z)-A^{N+1}h(z^{2^{A^{N+1}}})|
\le \frac{A^{N+1}B}{A-1}+
A^{N+1}B\sum_{k\ge 1}\delta(A)^k\le \frac{A^{N+1}}{2} 
$$
for large $A$.

Finally, given $\theta\in[0,2\pi)$, we construct  a sequence of points
$\{w_N=|w_N|e^{i\theta}\}$ at which  $u$ has extremal growth.  

Let $r_N(\theta)\in(\frac 16,\frac13)$ be such that
$$
h\bigl(r_N(\theta) e^{i\theta\cdot 2^{A^{N+1}}}\bigr)\ge 1,
$$
and let $w_N=r_N(\theta)^{2^{-A^{N+1}}}e^{i\theta}$.
Then
$$
2^{-A^{N}}\ge 1-|w_N| > 2^{-A^{N+1}},
$$
and
\begin{multline*}
u(w_N)\ge A^{N+1}h(w_N^{2^{A^{N+1}}})-\frac{A^{N+1}}{2}=
A^{N+1}h(r_N(\theta) e^{i\theta\cdot 2^{A^{N+1}}})-\frac{A^{N+1}}{2}\\ \ge 
\frac{A^{N+1}}{2}\ge c \log \frac 1 {1-|w_N|}.
\end{multline*}
\end{proof}


\begin{proof}[Proof of  Proposition \ref{prop:horowitz}.] 
 Let now  $\mu >1$ and an integer $\beta >1$  satisfy 
\eqref{eq:07b} and 
  $f_{\mu,\beta}$ be the corresponding  Horowitz function, see
\eqref{eq:07a}.  The zeros of $f_{\mu,\beta}$ are of the form
\beq
\label{eq:35}
z_{k,j}= \rho_k e^{i\pi \beta^{-k}}e^{2i\pi j \beta^{-k}}, \quad  k=1,2,\ldots ,  \  \ j=0,1,\ldots , \beta^k-1, 
\eeq
where  $\rho_k= \mu^{-\beta^{-k}}$.

 We will construct a sequence $\{\epsilon_k\}$ such that $\epsilon_k  \searrow 0$,
  \beq
 \label{eq:36}
      \sum_{k=1}^\infty \beta^k \epsilon_k < \infty,
 \eeq
 and, for some $a>0$, 
 \beq
 \label{eq:37}
  |f_{\mu,\beta}(z)|\ge\frac {\text{Const}}{(1-|z|)^a},  \
  \text{for} \ z\not \in \cup_{k,j} D_{k,j},
  \eeq    
 here $D_{k,j}=\{z\in \D; |z-z_{k,j}|<\epsilon_k\}$.
Since $\rho_k= \mu^{-\beta^{-k}}$, the discs $D_{k,j}$ are contained in the open unit disc; by \eqref{eq:36}, the sum of the lengths of the projections of
$\cup_jD_{k,j}$ on $\T$ is finite.  
Now Proposition~\ref{prop:horowitz} follows readily.  
 
 \smallskip

Consider the circle $|z|=r<1$  and  choose an integer $m$ such that  
\begin{equation}
\label{choice_m}
r^{\beta^m} \ge \mu^{-\sqrt{\beta}} > r^{\beta^{m+1}},
\end{equation} 
and hence, 
\[
\mu r^{\beta^{m-1}} \geq \mu^{1-1/\sqrt{\beta}} > 1 \quad  \text{and} \quad
\mu r^{\beta^{m+1}} < \mu^{1-\sqrt{\beta}} < 1.
\]
Relation \eqref{choice_m} also yields
\beq
\label{eq:37a}
 \left | m - \frac 1 {\log \beta} \log \frac 1 {1-r} \right | < C,
\eeq
where the constant $C$ does not depend on $r$.
One may assume  $m>1$. We have
 \beq
\label{eq:37aa}
f(z)=  \underbrace{\prod_{k=1}^{m-1}\left(1+\mu z^{\beta^k}\right)}_{P_m(z)}
 \left(1+\mu z^{\beta^m}\right)
 \underbrace {\prod_{k=m+1}^{\infty}\left(1+\mu z^{\beta^k}\right) }_{R_m(z)}.
 \eeq
 The   first factor is uniformly large on the circle $|z|=r$: 
\begin{multline*}
\log|P_m(z)|\ge \sum_{k=1}^{m-1} \log\left(\mu r^{\beta^k}-1\right)=
\\
\sum_{k=1}^{m-1}\left(\log\mu+\beta^k\log r\right)+
\sum_{k=1}^{m-1}\log\left(1-\frac1{\mu r^{\beta^k}}\right).
\end{multline*}
We get
\begin{multline*}
\sum_{k=1}^{m-1}\left(\log\mu+\beta^k\log r \right) 
\ge (m-1)\log \mu+\beta^m\log r\\
\ge (m-1)\log\mu-\sqrt{\beta}\log\mu
\end{multline*}
and 
\begin{multline*}
\sum_{k=1}^{m-1}\log\left(1-\frac1{\mu r^{\beta^k}}\right) \ge
\sum_{k=1}^{m-1}\log\left(1-\frac1{\mu r^{\beta^{m-1}}}\right)\\ \ge
 (m-1)\log\left(1-\frac1{\mu^{1-1/\sqrt{\beta}}}\right).
\end{multline*}
Relation  \eqref{eq:07b}  now yields
 \[
\log|P_m(z)|\ge
  (m-1)(\log\mu-1)-\sqrt{\beta}\log\mu,
  \]
 and, by \eqref{eq:37a},
 \[
 |P_m(z)|\geq \frac{\text{Const}}{(1-|z|)^a}, \qquad |z|=r, \
    a = \frac{\log \mu -1} {\log \beta}.
 \]
 
\smallskip
 
We apply the inequality $\log(1-x)\ge -\alpha x$, $x\le 1-\alpha ^{-1}$  
in order to prove that the third factor in \eqref{eq:37aa} is separated from zero when $|z|=r$: 
 \[
\log|R_m(z)|\ge \sum_{k=m+1}^\infty \log\left(1-\mu|z|^{\beta^k}\right)\ge -\alpha\sum_{k=m+1}^\infty  \mu|z|^{\beta^k}\ge c(\mu,\beta),
\] 
if we take $\alpha=\left ( 1-\mu^{1-\sqrt{\beta}} \right )^{-1}$, say.

\smallskip

The second factor in \eqref{eq:37aa}  vanishes at the points 
$\{z_{m,j}\}$, $j=0,1,\ldots \  ,$ $\beta^m-1$. 
Fix now  $q<1$ and let $\epsilon_m = q^m \beta^{-m}$. Condition \eqref{eq:36}  is then fulfilled.
It is also straightforward that
\[
|1+\mu z^{\beta^m}|\ge c q^m
\]
when $|z-z_{m,j}|=\epsilon_m$ for some $j$. Then the minimum principle implies the same inequality  whenever $\text{dist}(z,\{z_{m,j}\}_j )> \epsilon_m$.

It follows now from \eqref{eq:37a}  that, for
  any $a'<a$,  one can chose  $q$ sufficiently close to 1 such that
\[
  |f_{\mu,\beta}(z)|\ge\frac {\text{Const}}{(1-|z|)^{a'}},  \
  \text{for} \ z\not \in \cup_{k,j} D_{k,j}.
 \]   
  This completes the  proof of Proposition~\ref{prop:horowitz}
\end{proof}
  
\section{Positive harmonic functions}

In this section we prove Theorem~\ref{thm:positive}.  
First we prove that given a positive function $u\in \HH$, 
the set
\[
F_+(u)=\left\{\theta\in[0,2\pi): \limsup_{r\rightarrow
1}\frac{u(re^{i\theta})}{|\log(1-r)|}>0\right\}
\]
  is a countable union of sets with finite $ {H}_\lambda$ measure, with  the measuring function  
$\lambda(t)=t|\log t|$. This implies the first statement of 
Theorem~\ref{thm:positive}.

Let 
\[
F_{n}=\left\{
           \theta\in[0,2\pi): \limsup_{r\rightarrow 1}\frac{u(re^{i\theta})}{|\log(1-r)|}\ge\frac 2 n
                 \right\}.
\]
It suffices to prove  that 
$H_\lambda (F_n ) <\infty$ for all   $n$.

The function $u$ is positive and harmonic so it is the Poisson integral of a finite measure $\mu$ on $\T$. 
 Since $u\in\HH$ we have 
\beq
\label{eq:mu}
\mu(I)\le C|I|\log\frac{e}{|I|}
\eeq
for any arc $I$ on the unit circle (see  \cite{K}).  

In what follows we denote
\beqs
\label{eq:18}
\mu(\alpha, \beta)= \mu(\{e^{i\varphi}; \alpha \leq \varphi < \beta\}).
\eeqs
\begin{lemma}
\label{lem:1}
For each $n$ and each $\theta\in F_n$ there exists a decreasing sequence 
$\{\Delta_j\} $, $\Delta_j\rightarrow 0$ as $j\rightarrow \infty$ which  satisfies 
\beq
\label{eq:20}
\mu(\theta-\Delta_j,\theta+\Delta_j )  \ge   k \left (10\Delta_j  \log \frac 1 {10\Delta_j}  \right ),
\eeq 
with some $k>0$, depending on $C$ and  $n$ only. 
 \end{lemma}

Suppose this lemma is already proved. 
For each $\epsilon>0$ we can cover $F_n$ by intervals $I$ with centers on 
$F_n$ and of length less than $ \epsilon$ which satisfy 
$\mu(I)\ge k|5I|\bigl|\log|5I|\bigr|$, where $5I$ is the interval concentric with $I$
of length $5$ times that of $I$. By the Vitali lemma 
(see, for example, \cite[page 2]{He})
we can find a subfamily $I_j$ of disjoint intervals such that
$F_n\subset \cup_j 5I_j$.
We obtain 
\[
\sum_j|5I_j||\log|5I_j|\le \frac 1 {k}\sum_j\mu(I_j)\le \frac 1 {k} \mu(\T),
\] 
which yields ${H}_\lambda(F_n)\le\frac 1 {k}\mu(\T)<+\infty$.

\medskip

\begin{proof}[Proof of Lemma \ref{lem:1}]

 For $\theta \in F_{n}$  there exists a sequence $\{r_j\}_1^\infty$ such that
  $r_j\nearrow 1$ and 
\beq
 \label{eq:32a} 
 \frac 1 n \log \frac 1 {1-r_j} \leq u(r_je^{i\theta})= 
                \int_{-\pi}^{ \pi} P(r_je^{i\phi})d\mu(\theta-\phi)  .
   \eeq     
Let $a,A$ be two constants, such that $0<a<A$, their values will be determined below 
and $  \delta_j=a(1-r_j)$,  $ \Delta_j=A(1-r_j)$. 
By choosing $a$ sufficiently small and using (\ref{eq:mu}) we can achieve
   
\beq
\label{eq:32aa}
 \int_{-\delta_{j}}^{\delta_{j}} P(r_je^{i\phi})d\mu(\theta-\phi)\le
         \frac 1 {10 n}  \log \frac 1 {1-r_j} , \ j>j_0.
 \eeq         
 
Furthermore,  let 
  \[
Q(re^{i\phi})=-\partial_\phi P(re^{i\phi})=
       \frac{1}{2\pi}\frac{2r(1-r^2)\sin\phi}{(1-2r\cos\phi+r^2)^2}.
       \] 
 be  the 
 angular derivative of the Poisson kernel.  We then have   
   \begin{multline}    
   \label{eq:32ac}
   \int_{\delta_j<|\phi|\le \pi} P(r_je^{i\phi})d\mu(\theta-\phi) \le 
      \mu(\T)  +   \int_{\delta_j}^\pi \mu(\theta-\phi,\theta+ \phi) Q(r_je^{i\phi})d\phi  = \\
      \mu(\T) +  \int_{\delta_j}^{\Delta_j} \mu(\theta-\phi,\theta+ \phi) Q(r_je^{i\phi})d\phi  +
          \int_{\Delta_j}^\pi \mu(\theta-\phi,\theta+ \phi) Q(r_je^{i\phi})d\phi.     
    \end{multline}    
    In addition,
 \begin{multline*} 
  \int_{\Delta_j}^\pi \mu(\theta-\phi,\theta+ \phi) Q(r_je^{i\phi})d\phi \le
  C \log \frac e {2\Delta_j}  \int_{\Delta_j}^\pi 2 \phi   Q(r_je^{i\phi})d\phi \le \\
 2 C \log \frac 1 {1-r_j}  
     \left ( A(1-r_j)P\left (r_j e^{iA(1-r_j)} \right ) + \int_{A(1-r_j)}^\pi P(r_je^{i\phi}) d\phi 
     \right ).  
   \end{multline*}  
   Taking $A$ sufficiently large we obtain
   \beq
   \label{eq:32ab}
   \int_{\Delta_j}^\pi \mu(\theta-\phi,\theta+ \phi) Q(r_je^{i\phi})d\phi \le
       \frac 1 {10 n} \log  \frac 1 {1-r_j} , \ j>j_0.
 \eeq         
   
  It follows now from 
  \eqref{eq:32a}, 
  \eqref{eq:32aa}, \eqref{eq:32ac}, and 
  \eqref{eq:32ab}  that 
     \beq 
     \label{eq:32ae}
       \int_{\delta_j}^{\Delta_j} \mu(\theta-\phi,\theta+ \phi) Q(r_je^{i\phi})d\phi >
           \frac 1 {5 n} \log  \frac 1 {1-r_j} , \ j>j_0.
        \eeq     
        
    Integration by parts gives
     \[
        \int_{\delta_j}^{\Delta_j} \phi Q(r_je^{i\phi})d\phi
        \le
         \int_{0}^{\pi} \phi Q(r_je^{i\phi})d\phi  \le
           \int_{0}^{\pi}  P(r_je^{i\phi})d\phi  = \frac 12.
      \]                 
    This together with    \eqref{eq:32ae}  implies
    \[
     \int_{\delta_j}^{\Delta_j} \mu(\theta-\phi,\theta+ \phi) Q(r_je^{i\phi})d\phi >
             \frac 1 {5 n} \log  \frac 1 {1-r_j}
                    \int_{\delta_j}^{\Delta_j}  \phi Q(r_je^{i\phi})d\phi , \ j > j_0 .                  
      \]           
  Therefore, for each $j>j_0$ there exists $\phi_j \in (\delta_j, \Delta_j)$ such that             
   \[
   \mu(\theta-\phi_j, \theta+\phi_j) >
                       \frac { \phi_j} {5 n} \log  \frac 1 {1-r_j}.
    \]
The desired estimate \eqref{eq:20} follows.
\end{proof}

\medskip

To complete the proof of Theorem~\ref{thm:positive} we need to construct a positive harmonic function   $u\in \HH$ with $ H_\lambda(E_+(u))=\infty$,  where $\lambda(t)=t|\log  t|$.
Taking then its harmonic conjugate  $\tilde{u}$ we obtain the desired function as $f=\exp(u+i\tilde{u})$.
Clearly $D_+(f)=E_+(u)$, $f\in A^{-\infty}$, and  $|f|\ge 1$.

First we construct a function $v\in \HH$  such that  $ H_\lambda(E_+(v))>0$.
We   use a Cantor-type construction. 
 
Let $C_1$ be the union of two opposite quarters of the circle. We construct by induction sets $C_k\subset C_{k-1}$ such that $C_k$ consists of $2^{2^k-k}$ closed arcs of length $2^{1-2^k}\pi$ each. To obtain $C_k$ we divide each of the arcs of $C_{k-1}$ into $2^{2^{k-1}}$ equal subarcs and choose each second 
of them for $C_k$.  Denote 
$C=\cap C_k$ and consider the measures  $d\mu_k=2^k\chi(C_k)dt$, where
$\chi(C_k)$  is the characteristic function of $C_k$.

\begin{lemma}
\label{lem:4}
The sequence  $\{\mu_k\}$ converges weakly to a measure $\mu_0$ and $v=P*\mu_0\in\HH$.  In addition $C\subset E_+(v)$  and
$H_\lambda(C)>0$.  
\end{lemma}

\begin{proof}
We note that $\mu_k(\T)=2\pi$ for each $k$. 
Besides, for each arc $I$, with endpoints of the form 
$\exp(2\pi m i2^{-2^s})$, where $m$ is integer,  
the limit $\mu_k(I)$ as  $k\to \infty$ exists, just because 
all values   $\mu_k(I)$  are the same  when $k>s$. Now each continuous function on the circle can be uniformly approximated by  linear combinations of characteristic functions of such dyadic arcs. Thus for each continuous function $f$ there exists 
\[
\lim_{k\rightarrow\infty}\int_\T fd\mu_k
\]
and $\mu_k$ converge weakly to some positive measure $\mu_0$. 

In order to prove that $v= P*\mu_0 \in \HH$ it suffices to check that 
 \[
 \mu_0(J)\le \const |J|\log\frac{1}{|J|}
 \] 
for each arc   $J\subset \T$, and then to use again the results from \cite{K}. 
 
  Choose $s$ such that $2^{-2^{s}} 2\pi<|J|\le 2^{-2^{s-1}} 2\pi$.
 Now take an arc  $J_0\supset J$ with endpoints  of the form $\exp(2\pi m i2^{-2^{s}})$ with integer $m$ and such that  $|J_0|<3|J|$. We obtain 
 \[
\mu_k(J)\le \mu_k(J_0)=\mu_s(J_0)\le 2^s|J_0|<6|J||\log|J||.
\]
which is the desired inequality.

We now check that   $  C \subset E_+(v)$. We have
 \[
v(re^{i\alpha})=\int_{-\pi}^{\pi}P(re^{i\phi})d\mu_0(\alpha-\phi)\ge\int_0^\pi\mu_0(\alpha-\phi, \alpha+\phi)Q(re^{i\phi})d\phi.\]
Let  $\alpha\in \cap C_k=C$ and $2^{k-1}\le|\log(1-r)|< 2^{k}$. 
Then  \\
 $\mu_0(\alpha-\phi, \alpha+\phi)\ge c2^k\phi$ for $\phi<1-r$ and 
\[
v(re^{i\alpha})\ge
\int_0^{1-r}\!\mu_0(\alpha-\phi, \alpha+\phi)Q(re^{i\phi})d\phi\ge c2^k\int_0^{1-r}\phi Q(re^{i\phi})d\phi\ge c_12^k,
\]
when $r>r_0$. Thus $C=\cap C_k\subset E_+(v)$. Remind that $C_k$ is the union of $2^{2^k-k}$ arcs of length $2\pi 2^{-2^k}$  
and $C$ is a set of the type described in Theorem~\ref{th:cantor}. For $\lambda(t)=t|\log t|$ the theorem gives  $H_\lambda(C)\ge c>0$. 
\endproof

Finally we construct a sequence of measures $\mu^{(n)}$ and  sets $C^{(n)}$ such that $v^{(n)}=P*\mu^{(n)}$ is in $\HH$, 
$E_+(v^{(n)})\supset C^{(n)}$, 
$H_\lambda(C^{(n)}) \to \infty$ as  $n\to \infty$. 

The construction of $C^{(n)}$ is the following. Let $C^{(n)}_1$ be the union of $2^n$ arcs of length $2\pi2^{-n-1}$ (we divide the circle into $2^{n+1}$ equal arcs and take each second), $\mu^{(n)}_1=2^{-n+1}\chi(C^{(n)}_1)$. 
 Let $C^{(n)}_k$ be the union of $2^{2^{k-1}(n+1)-k}$ arcs of length $2\pi2^{-2^{k-1}(n+1)}$, then we divide each arc into equal arcs of length $2\pi2^{-2^{k}(n+1)}$ and take each second of them to form $C_{k+1}$. 
 We define also $\mu_k^{(n)}=2^{-n+k}\chi(C^{(n)}_k)$. As earlier the sequence of measures $\mu_k^{(n)}$ converges to a singular measure $\mu^{(n)}$ such that 
$$
\mu^{(n)}(J)\le \const\cdot 2^{-n}\cdot |J|\log\frac{1}{|J|}
$$ 
for every $n\ge 1$ and for every arc $J\subset \T$. 
Then $v^{(n)}=P*\mu^{(n)}\in \HH$, $u=\sum_{n\ge 1} v^{(n)}\in \HH$, and
$E_+(v^{(n)})\supset C^{(n)}$.
Theorem~\ref{th:cantor} shows that 
$H_\lambda(C^{(n)})\ge cn$,   
and, hence, $H_\lambda(E_+(u))=\infty$.  
\end{proof}

\section{  Hausdorff measure of Cantor sets}
In this section we prove Theorem \ref{th:cantor}. 
The left hand side inequality in \eqref{eq07}  is straightforward. 

We say that $C_s$ is the set of $s$'th generation, and the intervals $I^{(s)}_i$
of length $l_s$ that constitute $C_s$ are the intervals of $s$'th generation.
Denote the set of all these intervals by $\I_s$.

Let $\{J_j\}$ be a finite covering of $C$ by intervals of length less than $l_s$. 
We split the set $\{J_j\}$ into finitely many groups  
$\A_{s}, \A_{s+1},\ldots, \A_m$, where
\[
\A_p=\{J_j: l_{p+1}\le |J_j|< l_p\}.
\]
Some of these groups may be empty. Let 
\[
\M_{s+1}= \{ I \in \I_{s+1}; I\cap \left ( \cup_{J\in \A_{s}} J \right ) \neq \emptyset\}, \ {\rm{and}}\
M_{s+1}= \# \M_{s+1}.
\]
We have
\[
M_{s+1}\le \sum_{J\in \A_{s}} \left(\frac{|J|}{l_{s+1}}+1\right)\le
             2\sum_{J\in \A_{s}} \frac{|J|}{l_{s+1}}.
\]
Let $\RR_{s+1}= \I_{s+1}\setminus \M_{s+1}$ 
and $R_{s+1}=\#\RR_{s+1}=N_{s+1}-M_{s+1}$. We have $R_{s+1}$ intervals from $\I_{s+1}$  
which  do not intersect intervals from  $\A_s$. 

We continue the procedure. Take all intervals from $\I_{s+2}$ that are contained in 
$\cup_{I\in \RR_{s+1}} I$, the number of such intervals is $k_{s+1}R_{s+1}$. Let 
$M_{s+2}$ of them intersect  $\cup_{J\in \A_{s+1}} J $ and 
$R_{s+2}=k_{s+1}R_{s+1}-M_{s+2}$ be the number of remaining intervals.

After several steps we have $R_q$ intervals from $\I_q$ that intersect no 
interval from  $\A_s\cup \A_{s+1}...\cup \A_{q-1}$. Then we have 
$k_{q} R_q$ intervals in $\I_{q+1}$ that intersect no interval from  
$\A_s\cup \A_{s+1}\cup\ldots\cup \A_{q-1}$, and   
$M_{q+1}$ of them intersect intervals from $\A_q$, where
\beq
\label{eq:ha1}
M_{q+1}\le 2\sum_{J\in \A_{q}} \frac{|J|}{l_{q+1}}.
\eeq
Next we define $R_{q+1}=k_{q}R_q-M_{q+1}$. By induction 
\[
R_{q+1}=N_{q+1}-\sum_{r=s+1}^{q+1}\frac{M_r}{N_r}N_{q+1}.\]
If $R_{m+1}>0$, then we can find a point in $C$ that is not covered by the intervals from $\A_s\cup\ldots\cup \A_m$. Thus $R_{m+1}= 0$,   
and we get
\beq
\label{eq:ha2}
\sum_{r=s+1}^{m+1}\frac{M_r}{N_r}= 1.  
\eeq

Set $b_s=\inf_{q\ge s}N_q\lambda(l_q)$.
 Now we  use \eqref{eq:cantorh}, \eqref{eq:ha1}, \eqref{eq:ha2}  to estimate the sum of
  $\lambda(|J_j|)$: 
\begin{gather*}
\sum_j\lambda(|J_j|)=\sum_{p=s}^m\sum_{J\in \A_p}\lambda(|J|)\ge a\sum_{p=s}^m\sum_{J\in \A_p}\lambda(l_{p+1})\frac{|J|}{l_{p+1}}\ge\\
ab_s\sum_{p=s}^m\frac{1}{N_{p+1}}\sum_{J\in \A_p}\frac{|J|}{l_{p+1}}
\ge \frac{ab_s}{2}\sum_{p=s}^m\frac{M_{p+1}}{N_{p+1}}=  
\frac{a}{2}b_s,
\end{gather*}
for any finite cover of $C$ with $|J_j|<l_s$. This shows that 
\[
H_\lambda(C)\ge 
\frac{a}{2}\liminf_{s\rightarrow\infty} N_s\lambda(l_s).
\]

\section*{Acknowledgments}
This work was done  when Yu.L. and E.M.  visited
the Universities   of Aix-Marseille and Paul Sabatier. They thank
the Universities  for hospitality and support.


\end{document}